\documentclass{amsart}

\usepackage[english]{babel}
\usepackage[utf8]{inputenc}

\usepackage[colorlinks]{hyperref}
\hypersetup{allcolors=magenta}
\usepackage{mathrsfs}
\usepackage{amsmath}
\usepackage{amsfonts}
\usepackage{amssymb}
\usepackage{amsthm}
\usepackage{amsrefs}
\usepackage{enumerate}
\usepackage{float}

\usepackage{array}

\usepackage[skip=10pt]{caption}

\newtheorem{thm}{Theorem}[section]
\newtheorem{rmk}[thm]{Remark}
\newtheorem{prop}[thm]{Proposition}
\newtheorem{cor}{Corollary}[thm]
\newtheorem{lema}[thm]{Lemma}
\newtheorem{defi}[thm]{Definition}

\everymath{\displaystyle}

\emergencystretch=\maxdimen
\title[Stability of pencils of plane sextics]{Stability of pencils of plane sextics and Halphen pencils of index two}
\author{Aline Zanardini}
\address{Department of Mathematics, University of Pennsylvania, USA}
\email{alinez@math.upenn.edu}
\date{\today}

\begin{document}

\begin{abstract}
We study the stability of pencils of plane sextics in the sense of geometric invariant theory. In particular, we obtain a complete and geometric description of the stability of Halphen pencils of index two.
\end{abstract}

\maketitle

\tableofcontents

\section{Introduction}

In this paper we study the stability of pencils of plane curves of degree six under projective equivalence in the sense of geometric invariant theory (GIT). Building in the results we obtained in \cite{azconstr} and \cite{azstabd}, we completely describe the stability of Halphen pencils of index two -- classical geometric objects first introduced by Halphen in 1882 \cite{halphen}. These are pencils of plane curves of degree six with exactly nine base points (possibly infinitely near) of multiplicity two (Definition \ref{defHalphen}).

More generally, a Halphen pencil (of index $m$) corresponds to a rational surface $Y$ that admits a fibration $f: Y \to \mathbb{P}^1$ with generic fiber a smooth curve of genus one (see e.g. \cite{dc}*{Chapter V, \S 6}). These surfaces are called rational elliptic surfaces (of index $m$), and the correspondence is given by the blow-up of $\mathbb{P}^2$ at the nine base points of the pencil. They have exactly one multiple fiber of multiplicity $m$, and $f$ does not have a section if and only if $m>1$. In this paper we are interested in the case $m=2$.

We follow an approach similar to that of Miranda in \cite{stab}, and we provide a complete and geometric characterization of the stability of Halphen pencils of index two in terms of the types of singular fibers appearing in the associated rational elliptic surfaces (of index two). In \cite{stab}, Miranda describes the stability of pencils of plane cubics up to projective equivalence, which leads to a compactification of the moduli space of rational elliptic surfaces with section. Such compactification agrees with the one obtained by the same author in \cite{mirW}, where the surfaces are described by their Weierstrass models.  When the existence of a global section is not assumed, there is no analogue for the Weierstrass model and even the dimensions of the parameter spaces involved are much higher, which makes describing the stability conditions a much harder problem. The results we obtained in \cite{azconstr} and \cite{azstabd}  have helped us to overcome such difficulties. 
 
The main results of this paper are given by Theorems \ref{m1}, \ref{m2} and \ref{m3} below, where $\mathcal{P}$ denotes a Halphen pencil of index two and $Y$ denotes the associated rational elliptic surface. We write the pencil $\mathcal{P}$ as $\lambda B + \mu (2C)=0$, where $C$ is the unique cubic through the nine base points of $\mathcal{P}$ and the curve $B$ corresponds to some (non-multiple) fiber, which we denote by $F$.

\begin{thm}
When $C$ is smooth the pencil $\mathcal{P}$ is stable if and only if one of the following statements hold
\begin{enumerate}[(i)]
\item all fibers of $Y$ are reduced 
\item $Y$ contains at most one non-reduced fiber $F$ of type $I_n^*$ or $IV^*$ 
\item there exists exactly one (non-multiple) fiber $F$ in $Y$ such $lct(Y,F)\leq 1/4$ and $B$ is semistable 
\item $Y$ contains two fibers of type $I_0^*$ and there is no one-parameter subgroup $\lambda$ that destabilizes the two corresponding curves simultaneously.
\end{enumerate}
\label{m1}
\end{thm}

\begin{thm}
When $C$ is singular the pencil $\mathcal{P}$ is stable if and only if  one of the following statements hold
\begin{enumerate}[(i)]
\item all fibers of $Y$ are reduced 
\item $\mathcal{P}$ contains at worst two strictly semistable curves and there is no one-parameter subgroup $\lambda$ that destabilizes these two curves simultaneously
\item $Y$ contains a fiber of type $IV^*$ and $B$ is unstable
\end{enumerate}
\label{m2}
\end{thm}

\begin{thm}
The pencil $\mathcal{P}$ is semistable if and only if every curve in $\mathcal{P}$ is semistable or $Y$ does not contain a fiber $F$ of type $II^*$.
\label{m3}
\end{thm}

We observe that every pencil of plane curves can also be seen as a holomorphic foliation of $\mathbb{P}^2$, in particular any Halphen pencil of index two. In \cite{sad}, the authors give a criterion for when a foliation of $\mathbb{P}^2$ of degree seven is a Halphen pencil of index two. In \cite{ca4}, a stratification of the space of foliations on $\mathbb{P}^2$ of degree $d$ is given  (in particular $d=7$), thus providing an alternative description of the semistable and unstable Halphen pencils of index two under projective equivalence -- viewed as foliations. However, such stratification does not provide a description of the stable pencils, which we do in this paper.

\subsection*{Organization}

The paper is organized as follows: In Section \ref{sc} we describe the vanishing criteria in terms of Pl\"{u}cker coordinates for (semi)stability of pencils of plane sextics. Then, a more geometric description is presented in Section \ref{gd}, where we translate these vanishing criteria into equations for the generators of the pencil. Finally, in Section \ref{halphen} we provide a complete characterization of the (semi)stability of Halphen pencils of index two.

\section{Stability Criterion for Pencils of Plane Sextics}
\label{sc}

Following the ideas in \cite{stab}, we view a pencil of plane sextics as a choice of a line in the space of all plane curves of degree six. In other words, we identify the space $\mathscr{P}_6$ of all such pencils with the Grassmannian $Gr(2,S^6V^{\ast})$, where  $V \doteq H^0(\mathbb{P}^2,\mathcal{O}_{\mathbb{P}^2}(1))$. The latter, in turn, can be embedded in $\mathbb{P}(\Lambda^2 S^6V^{\ast})$ via Pl\"{u}cker coordinates so that a pencil corresponds to a decomposable $2-$form. The automorphism group of $\mathbb{P}^2$, $PGL(3)$, acts naturally on $V$, hence on the invariant subvariety $\mathscr{P}_6$.

Our main tool for determining which are the stable and strictly semistable pencils is the numerical criterion of Hilbert-Mumford. Recall that the Hilbert-Mumford criterion says a pencil $\mathcal{P} \in \mathscr{P}_6$ is unstable (resp. not stable) if and only if there exists a 1-parameter subgroup $\mathbb{C}^{\times} \to SL(V)$ with respect to which all weights are positive (resp.  non-negative). Thus, we need to obtain explicit vanishing conditions on the Pl\"{u}cker coordinates and for that we need to know how the diagonal elements in $PGL(V)$ act on these coordinates. In fact since $PGL(V)$ and $SL(V)$ act with the same orbits it is convenient to focus on the diagonal action of the latter.

If we choose a pencil $\mathcal{P} \in \mathscr{P}_6$ and two curves $C_f$ and $C_g$ as generators, these represented by $f=\sum f_{ij}x^iy^jz^{6-i-j}$ and $g=\sum g_{ij}x^iy^jz^{6-i-j}$ (respectively), then the Pl\"{u}cker embedding takes the $2\times 27$ matrix $\begin{pmatrix} f_{ij} \\ g_{ij} \end{pmatrix}$ to the point in $\mathbb{P}^{{28\choose 2}-1}$ whose coordinates are given by all the $2\times 2$ minors $m_{ijkl}\doteq \begin{vmatrix} f_{ij} & f_{kl} \\ g_{ij} & g_{kl} \end{vmatrix}$. 

Now, if and element of $SL(V)$ is given by 
\[
\begin{pmatrix} \alpha & 0 & 0\\ 0 & \beta & 0\\ 0 & 0 & \gamma \end{pmatrix}
\]
in some choice of coordinates $[x,y,z]$ in $\mathbb{P}^2$, then its action on the coordinates of $\mathbb{P}^2$ is given by $[x,y,z] \mapsto [\alpha x,\beta y, \gamma z]$. Thus, the action on a point $(f_{ij})\in S^6 V^*$ representing $\sum f_{ij}x^iy^jz^{6-i-j}$ is given by $(f_{ij}) \mapsto (\alpha^i\beta^j\gamma^{6-i-j}f_{ij})$ and, therefore, it acts on the Pl\"{u}cker coordinates  by 
\[
(m_{ijkl}) \mapsto (\alpha^{i+k}\beta^{j+l}\gamma^{12-i-j-k-l}m_{ijkl})
\]

So we can now express the Hilbert-Mumford criterion for a pencil $\mathcal{P} \in \mathscr{P}_6$ as the vanishing of some of its Pl\"{u}cker coordinates $(m_{ijkl})$ with respect to a convenient choice of basis. We will show that there are only (explicit) finitely many critical one-parameter subgroups $\lambda$ that need to be tested when analyzing the positivity of the Hilbert-Mumford weight $\mu(\mathcal{P},\lambda)$  (defined in (\ref{hmw}) below). 

We may assume any one-parameter subgroup $\lambda$ is normalized, meaning we choose coordinates $[x,y,z]$ in $\mathbb{P}^2$ so that we have
\begin{eqnarray*}
\lambda: \mathbb{C}^{\times}  &\to& SL(V)\\
t &\mapsto&\left( [x,y,z] \mapsto \begin{pmatrix} t^{a_x} & 0 & 0 \\ 0 & t^{a_y} & 0 \\ 0 & 0 & t^{a_z} \end{pmatrix}\cdot \begin {pmatrix}x\\y\\z \end{pmatrix}\right)
\end{eqnarray*}
for some weights $a_x,a_y,a_z \in \mathbb{Z}$ with $a_x \geq a_y \geq a_z, a_x>0$ and $a_x+a_y+a_z=0$.

In particular, the action of $\lambda(t)$ in the Pl\"{u}cker coordinates is given by 
\[
(m_{ijkl}) \mapsto (t^{e_{ijkl}} m_{ijkl})
\]
where $e_{ijkl} \doteq a_x(2i+2k+j+l-12)+a_y(2j+2l+i+k-12)$.

Further, we often normalize the weights so that $a_x=1,a_y=a$ and $a_z=-1-a$ for some $a\in [-1/2,1]\cap \mathbb{Q}$. Then 
\[
e_{ijkl}=e_{ijkl}(a)= (2i+2k+j+l-12)+a(2j+2l+i+k-12)
\]

The sign of the function $\mu(\mathcal{P},\lambda)$ does not change under these reductions and the Hilbert-Mumford criterion becomes:

\begin{prop}
A pencil $\mathcal{P}\in \mathscr{P}_6$ is unstable (resp. not stable) if and only if there exists a rational number $a \in[-1/2,1]$ and coordinates in $\mathbb{P}^2$ such that if the pencil is represented  in those coordinates by $(m_{ijkl})$, then $m_{ijkl}=0$ whenever $e_{ijkl}(a)\leq 0$ (resp. $e_{ijkl}(a) < 0$).
\label{hm}
\end{prop}

A priori, for each choice of coordinates in $\mathbb{P}^2$ one would need to test all possible values of $a\in[-1/2,1]\cap \mathbb{Q}$ to verify the stability criterion. Because the function  (for a fixed $\mathcal{P}$ and a choice of coordinates) 
\begin{equation}
\mu(\mathcal{P},\lambda) \doteq \min\{e_{ijkl}(a)\,\,:\,\,m_{ijkl}\neq 0\}
\label{hmw}
\end{equation}
is piecewise linear, a key observation is that we only need to test its positivity for a finite number of critical values $a\in[-1/2,1]\cap \mathbb{Q}$.

In other words, the conditions $e_{ijkl}(a)\leq 0$ (resp. $e_{ijkl}(a) < 0$) divide the interval $[-1/2,1]$ into finitely many subintervals $[a_n,a_{n+1}]$ within which the truthfulness of the inequality remains constant. That is, for each interval $[a_n,a_{n+1}]$ we can find values of $i,j,k$ and $l$ for which the inequality $e_{ijkl}(a)\leq 0$  (resp. $e_{ijkl}(a) < 0$) remains true for all $a\in [a_n,a_{n+1}]$.

To find these intervals we proceed as follows. For computational reasons we first let $r=i+k$ and $s=j+l$. Then, for each possible pair $(r,s)$ in the set 
\[
\{(r,s)\in \{0,1,\ldots,12\}\times\{0,1,\ldots,12\}\,\,:\,\, r+s \leq 12\},
\]
we test whether we can solve the inequality $2r+s-12+a(2s+r-12)\leq 0$ (resp. $<0$) for the variable $a$ imposing the restriction $a\in [-1/2,1]$. 

The intervals we find are given by Lemmas \ref{uint} and \ref{nsint} below:

\begin{lema}
The condition $e_{ijkl}(a)\leq 0$ divides the interval $[-1/2,1]$ into finitely many subintervals and in order to obtain minimal conditions for unstability, it suffices considering only the following six distinct subintervals:
\[
(-1/3,-2/7),(-2/7,-1/5),(-1/11,0),(1/7,1/4),(1/4,2/5),(1/2,1)
\]
\label{uint}
\end{lema}

\begin{lema}
The condition $e_{ijkl}(a)<0$ divides the interval $[-1/2,1]$ into finitely many subintervals and the subintervals that give (distinct) minimal conditions for non-stability are such that it suffices taking $a\in\{-1/2,-2/7,-1/5,0,1/4,2/5,1\}$. 
\label{nsint}
\end{lema}

In particular, we can restate the criteria for unstability ( resp. non-stability) as in Theorem \ref{unstable1} (resp. Theorem \ref{notstable}):

\begin{thm}
A pencil $\mathcal{P}\in \mathscr{P}_6$ is unstable if and only if there exist coordinates in $\mathbb{P}^2$ so that if the pencil is represented  in those coordinates by $(m_{ijkl})$, then $m_{ijkl}=0$ whenever the (appropriate) values of $i,j,k$ and $l$ satisfy either one of the following conditions:
\begin{enumerate}
\item $(2i+2k+j+l-12)-13/42(2j+2l+i+k-12)\leq 0$
\item $(2i+2k+j+l-12)-8/35(2j+2l+i+k-12)\leq0$
\item $(2i+2k+j+l-12)-1/12(2j+2l+i+k-12)\leq 0$
\item $(2i+2k+j+l-12)+3/14(2j+2l+i+k-12)\leq 0$
\item $(2i+2k+j+l-12)+3/10(2j+2l+i+k-12)\leq0$
\item $(2i+2k+j+l-12)+3/4(2j+2l+i+k-12)\leq 0$
\end{enumerate}
\label{unstable1}
\end{thm}

\begin{thm}
A pencil $\mathcal{P}\in \mathscr{P}_6$ is not stable if and only if there exist coordinates in $\mathbb{P}^2$ so that if the pencil is represented  in those coordinates by $(m_{ijkl})$, then $m_{ijkl}=0$ whenever the (appropriate) values of $i,j,k$ and $l$ satisfy either one of the following conditions:
\begin{enumerate}
\item $(2i+2k+j+l-12)-1/2(2j+2l+i+k-12)<0$
\item $(2i+2k+j+l-12)-2/7(2j+2l+i+k-12)<0$
\item $(2i+2k+j+l-12)-1/5(2j+2l+i+k-12)<0$
\item $(2i+2k+j+l-12)<0$
\item $(2i+2k+j+l-12)+1/4(2j+2l+i+k-12)<0$
\item $(2i+2k+j+l-12)+2/5(2j+2l+i+k-12)<0$
\item $(2i+2k+j+l-12)+(2j+2l+i+k-12)<0$
\end{enumerate}
\label{notstable}
\end{thm}

Now, in order to know what is the set of values $i,j,k$ and $l$ for which the Pl\"{u}cker coordinates $m_{ijkl}$ vanish in Theorems \ref{unstable1} and \ref{notstable} above, it is convenient to express these values in terms of  the pairs $(r,s)$. 

For each pair $(r,s)$ we let $M_{rs}\doteq \{m_{ijkl}\,\,:\,\,i+k=r\,\,\mbox{and}\,\,j+l=s\}$ and we obtain the following:

\begin{thm}
A pencil $\mathcal{P}\in \mathscr{P}_6$ is unstable if and only if there exist coordinates in $\mathbb{P}^2$ so that if the pencil is represented  in those coordinates by $(m_{ijkl})$, then either
\begin{enumerate}
\item  $M_{rs}=\{0\}$ for all the pairs $(r,s)$ in the list below
\[
\begin{matrix}
(0,0) &(0,1) &(0,2) &(0,3) & (0,4) &(0,5) & (0,6) & (0,7) & (0,8) & (0,9)\\
(0,10) & (0,11) & (0,12) & (1,0) &(1,1) &(1,2) &(1,3) & (1,4) &(1,5) & (1,6)\\
(1,7) & (1,8) & (1,9) & (1,10) & (1,11) & (2,0) &(2,1) &(2,2) &(2,3) & (2,4)\\
(2,5) & (2,6) & (2,7) & (2,8) & (2,9) & (2,10) &  (3,0) &(3,1) &(3,2) &(3,3)\\
(3,4) & (3,5) & (3,6) & (3,7) & (3,8) & (4,0) &(4,1) &(4,2) &(4,3) & (4,4)
\end{matrix}
\]
and a number $a\in (-1/3,-2/7)$ will exhibit $\mathcal{P}$ as unstable; or
\item $M_{rs}=\{0\}$ for all the pairs $(r,s)$ in the list below:
\[
\begin{matrix}
(0,0) &(0,1) &(0,2) &(0,3) & (0,4) &(0,5) & (0,6) & (0,7) & (0,8) & (0,9)\\
(0,10) & (0,11) & (0,12) & (1,0) &(1,1) &(1,2) &(1,3) & (1,4) &(1,5) & (1,6)\\
(1,7) & (1,8) & (1,9) & (1,10) & (1,11) & (2,0) &(2,1) &(2,2) &(2,3) & (2,4)\\
(2,5) & (2,6) & (2,7) & (2,8) & (2,9) & (2,10) &  (3,0) &(3,1) &(3,2) &(3,3)\\
(3,4) & (3,5) & (3,6) & (3,7) & (4,0) &(4,1) &(4,2) &(4,3) & (4,4) & (5,0)
\end{matrix}
\]
and a number $a\in (-2/7,-1/5)$ will exhibit $\mathcal{P}$ as unstable; or
\item $M_{rs}=\{0\}$ for all the pairs $(r,s)$ in the list below
\[
\begin{matrix}
(0,0) &(0,1) &(0,2) &(0,3) & (0,4) &(0,5) & (0,6) & (0,7) & (0,8) & (0,9)\\
(0,10) & (0,11) & (0,12) & (1,0) &(1,1) &(1,2) &(1,3) & (1,4) &(1,5) & (1,6)\\
(1,7) & (1,8) & (1,9) & (1,10) &  (2,0) &(2,1) &(2,2) &(2,3) & (2,4) & (2,5)\\
 (2,6) & (2,7) & (2,8) &   (3,0) &(3,1) &(3,2) &(3,3) & (3,4) & (3,5) & (3,6)\\
(4,0) &(4,1) &(4,2) &(4,3) & (4,4) & (5,0) & (5,1)&&&
\end{matrix}
\]
and a number $a\in (-1/11,0)$ will exhibit $\mathcal{P}$ as unstable; or
\item $M_{rs}=\{0\}$ for all the pairs $(r,s)$ in the list below 
\[
\begin{matrix}
(0,0) &(0,1) &(0,2) &(0,3) & (0,4) &(0,5) & (0,6) & (0,7) & (0,8) & (0,9)\\
(0,10) & (1,0) &(1,1) &(1,2) &(1,3) & (1,4) &(1,5) & (1,6) &(1,7) & (1,8)\\
 (2,0) &(2,1) &(2,2) &(2,3) & (2,4) & (2,5)& (2,6) & (2,7) &    (3,0) &(3,1)\\
(3,2) &(3,3) & (3,4) & (3,5) & (4,0) &(4,1) &(4,2) &(4,3) & (4,4) & (5,0)\\
(5,1)&(5,2)&(6,0)& &&&&&&
\end{matrix}
\]
and a number $a\in (1/7,1/4)$ will exhibit $\mathcal{P}$ as unstable; or
\item $M_{rs}=\{0\}$ for all the pairs $(r,s)$ in the list below
\[
\begin{matrix}
(0,0) &(0,1) &(0,2) &(0,3) & (0,4) &(0,5) & (0,6) & (0,7) & (0,8) & (0,9)\\
(1,0) &(1,1) &(1,2) &(1,3) & (1,4) &(1,5) & (1,6) &(1,7) & (1,8)& (2,0)\\
(2,1) &(2,2) &(2,3) & (2,4) & (2,5)& (2,6) &    (3,0) &(3,1) &
(3,2) &(3,3)\\
 (3,4) & (3,5) & (4,0) &(4,1) &(4,2) &(4,3) & (4,4) & (5,0) &
(5,1)&(5,2)\\
(6,0)& (6,1) &&&&&&&&
\end{matrix}
\]
and a number $a\in(1/4,2/5)$ will exhibit $\mathcal{P}$ as unstable; or
\item $M_{rs}=\{0\}$ for all the pairs $(r,s)$ in the list below
\[
\begin{matrix}
(0,0) &(0,1) &(0,2) &(0,3) & (0,4) &(0,5) & (0,6) & (0,7) & (0,8) & (1,0)\\
(1,1) &(1,2) &(1,3) & (1,4) &(1,5) & (1,6) &(1,7) & (2,0) &
(2,1) &(2,2)\\
(2,3) & (2,4) & (2,5)& (2,6) &    (3,0) &(3,1) &
(3,2) &(3,3) & (3,4) & (3,5)\\
(4,0) &(4,1) &(4,2) &(4,3) & (4,4) & (5,0) &
(5,1)&(5,2) &(6,0)& (6,1)\\
(7,0)&&&&&&&&&
\end{matrix}
\]
and a number $a\in(1/2,1)$ will exhibit $\mathcal{P}$ as unstable.
\end{enumerate}
\label{unstable2}
\end{thm}

\begin{thm}
A pencil $\mathcal{P}\in \mathscr{P}_6$ is not stable if and only if there exist coordinates in $\mathbb{P}^2$ so that if the pencil is represented  in those coordinates by $(m_{ijkl})$, then either
\begin{enumerate}
\item $M_{rs}=\{0\}$ for all the pairs $(r,s)$ in the list below
\[
\begin{matrix}
(0,0) & (0,1) & (0,2) & (0,3) & (0,4) & (0,5) & (0,6) & (0,7) & (0,8) & (0,9)  \\
(0,10) & (0,11) & (0,12) & (1,0) & (1,1) & (1,2) & (1,3) & (1,4) & (1,5) & (1,6) \\
 (1,7) & (1,8) & (1,9) & (1,10) & (1,11) & (2,0) & (2,1) & (2,3) & (2,4) & (2,5)\\
(2,6) & (2,7) & (2,8)  & (2,9)  & (2,10)  & (3,0) & (3,1)  & (3,2)  & (3,3) & (3,4) \\
(3,5) & (3,6)  & (3,7)  & (3,8) & (3,9)  &   &  &  & &
\end{matrix}
\]
and $a=-1/2$ will exhibit $\mathcal{P}$ as not stable; or
\item $M_{rs}=\{0\}$ for all the pairs $(r,s)$ in the list below
\[
\begin{matrix}
(0,0) & (0,1) & (0,2) & (0,3) & (0,4) & (0,5) & (0,6) & (0,7) & (0,8) & (0,9)  \\
(0,10) & (0,11) & (0,12) & (1,0) & (1,1) & (1,2) & (1,3) & (1,4) & (1,5) & (1,6) \\
 (1,7) & (1,8) & (1,9) & (1,10) & (1,11) & (2,0) & (2,1) & (2,3) & (2,4) & (2,5)\\
(2,6) & (2,7) & (2,8)  & (2,9)  & (2,10)  & (3,0) & (3,1)  & (3,2)  & (3,3) & (3,4) \\
(3,5) & (3,6)  & (3,7)  & (4,0) & (4,1)  & (4,2)  & (4,3)  &  & &
\end{matrix}
\]
and $a=-2/7$ will exhibit $\mathcal{P}$ as not stable; or
\item $M_{rs}=\{0\}$ for all the pairs $(r,s)$ in the list below
\[
\begin{matrix}
(0,0) & (0,1) & (0,2) & (0,3) & (0,4) & (0,5) & (0,6) & (0,7) & (0,8) & (0,9)  \\
(0,10) & (0,11) & (0,12) & (1,0) & (1,1) & (1,2) & (1,3) & (1,4) & (1,5) & (1,6) \\
 (1,7) & (1,8) & (1,9) & (1,10) & (1,11) & (2,0) & (2,1) & (2,3) & (2,4) & (2,5)\\
(2,6) & (2,7) & (2,8)  & (2,9)  &  (3,0) & (3,1)  & (3,2)  & (3,3) & (3,4) & (3,5)\\
(3,6)  & (4,0) & (4,1)  & (4,2)  & (4,3)  &  (5,0) & & & &
\end{matrix}
\]
and $a=-1/5$ will exhibit $\mathcal{P}$ as not stable; or
\item $M_{rs}=\{0\}$ for all the pairs $(r,s)$ in the list below
\[
\begin{matrix}
(0,0) & (0,1) & (0,2) & (0,3) & (0,4) & (0,5) & (0,6) & (0,7) & (0,8) & (0,9)  \\
(0,10) & (0,11) &  (1,0) & (1,1) & (1,2) & (1,3) & (1,4) & (1,5) & (1,6) & (1,7)\\
(1,8) & (1,9) &   (2,0) & (2,1) & (2,3) & (2,4) & (2,5) & (2,6) & (2,7)&  (3,0)\\
 (3,1)  & (3,2)  & (3,3) & (3,4) & (3,5) & (4,0) & (4,1)  & (4,2)& (4,3)  &  (5,0)\\
(5,1) & & & &&&&&&
\end{matrix}
\]
and $a=0$ will exhibit $\mathcal{P}$ as not stable; or
\item $M_{rs}=\{0\}$ for all the pairs $(r,s)$ in the list below
\[
\begin{matrix}
(0,0) & (0,1) & (0,2) & (0,3) & (0,4) & (0,5) & (0,6) & (0,7) & (0,8) & (0,9)  \\
 (1,0) & (1,1) & (1,2) & (1,3) & (1,4) & (1,5) & (1,6) & (1,7) &
(1,8) &   (2,0)\\
 (2,1) & (2,3) & (2,4) & (2,5) & (2,6) &   (3,0) & (3,1)  & (3,2)  & (3,3) & (3,4)\\
(3,5) & (4,0) & (4,1)  & (4,2)& (4,3)  &  (5,0) & (5,1) & (5,2) & (6,0) & 
\end{matrix}
\]
and $a=1/4$ will exhibit $\mathcal{P}$ as not stable; or
\item $M_{rs}=\{0\}$ for all the pairs $(r,s)$ in the list below
\[
\begin{matrix}
(0,0) & (0,1) & (0,2) & (0,3) & (0,4) & (0,5) & (0,6) & (0,7) & (0,8) & (0,9)  \\
 (1,0) & (1,1) & (1,2) & (1,3) & (1,4) & (1,5) & (1,6) & (1,7) &
 (2,0) & (2,1)\\
 (2,3) & (2,4) & (2,5) & (2,6) &   (3,0) & (3,1)  & (3,2)  & (3,3) & (3,4) & (3,5)\\
(4,0) & (4,1)  & (4,2)& (4,3)  &  (5,0) & (5,1) & (5,2) & (6,0) & (6,1) & 
\end{matrix}
\]
and $a=2/5$ will exhibit $\mathcal{P}$ as not stable; or
\item $M_{rs}=\{0\}$ for all the pairs $(r,s)$ in the list below
\[
\begin{matrix}
(0,0) & (0,1) & (0,2) & (0,3) & (0,4) & (0,5) & (0,6) & (0,7)  &  (1,0) & (1,1)\\
(1,2) & (1,3) & (1,4) & (1,5) & (1,6) &
 (2,0) & (2,1)& (2,3) & (2,4) &  (2,5)\\
(3,0) & (3,1)  & (3,2)  & (3,3) & (3,4) &
(4,0) & (4,1)  & (4,2)& (4,3) & (5,0)\\
 (5,1) & (5,2) & (6,0) & (6,1) & (7,0) &&&&&
\end{matrix}
\]
and $a=1$ will exhibit $\mathcal{P}$ as not stable.
\end{enumerate}
\label{notstable2}
\end{thm}

\section{A Geometric Description}
\label{gd}

In the previous section we have completely characterized the stability of a pencil $\mathcal{P}\in \mathscr{P}_6$ in terms of its Pl\"{u}cker coordinates $(m_{ijkl})$. We now want to understand which are the geometric properties  unstable and non-stable pencils have. More precisely, we want to translate the stability criteria into equations for the generators of the pencil. 

Throughout this section, given an unstable (resp. not stable) pencil  $\mathcal{P}\in \mathscr{P}_6$ we choose coordinates $[x,y,z]$ in $\mathbb{P}^2$ as in Theorem \ref{unstable1} (resp. \ref{notstable}) and generators $C_f$ and $C_g$ having defining polynomials (in these coordinates) $f=\sum f_{ij}x^iy^jz^{6-i-j}$ and $g=\sum g_{ij}x^iy^jz^{6-i-j}$. Then, the idea is that each vanishing condition $m_{ijkl}=0$ translates into the vanishing of the coefficients of some pair $C_{f'}$ and $C_{g'}$ of generators (not necessarily the original pair). 

To illustrate what kind of computations are involved in this process we prove Theorem \ref{gennotstable} below. We use the notation $\langle m_1,\ldots,m_n \rangle$ to denote the subspace of homogeneous polynomials of degree six in the variables $x,y$ and $z$ which is generated by the monomials $m_i$. Whereas $\rangle m_1,\ldots,m_n \langle$ denotes the subspace of those polynomials which are generated by all the monomials which are different from the $m_i$.

\begin{thm}
A pencil $\mathcal{P}\in \mathscr{P}_6$ satisfies the vanishing conditions in case $1$ of Theorem \ref{notstable2} if and only if there exist coordinates in $\mathbb{P}^2$ and generators $C_f$ and $C_g$ of $\mathcal{P}$ such that either
\begin{enumerate}[{Case} 1]
\item $f \in \langle x^4z^2,x^4yz,x^4y^2,x^5z,x^5y,x^6 \rangle$ and $g$ is arbitrary
\item $f \in \langle x^3z^3,x^3yz^3,x^3y^2z,x^3y^3,x^4z^2,x^4yz,x^4y^2,x^5z,x^5y,x^6 \rangle$ 
\subitem and $g\in \rangle z^6,yz^5,y^2z^4,y^3z^3,y^4z^2,y^5z,y^6 \langle$ 
\item $f$ and $g \in \langle  x^iy^jz^{6-i-j} \rangle$, where $2\leq i \leq 6, 0\leq j \leq 6$ and $i+j\leq 6$ 
\end{enumerate}
\label{gennotstable}

\end{thm}

\begin{proof}
Let us assume $\mathcal{P}$ is not stable and that its Pl\"{u}cker coordinates $(m_{ijkl})$ must vanish for all $i,j,k$ and $l$ satisfying $i+k=r$ and $j+l=s$ for all the pairs $(r,s)$ in case $1$ of Theorem \ref{notstable2}. Using the relations $m_{ijkl}=-m_{klij}$ and $m_{ijij}=0$ we can compute the minimal set of values $\{i,j,k,l\}$ (in order) so that the $m_{ijkl}$ vanish. 

In other words, we find all integers $i,j,k$ and $l$  subject to the restrictions
\begin{enumerate}[(i)]
 \item $0\leq i,j,k,l \leq 6$,
\item $ i+j \leq 6$,
\item $k+l\leq 6$, and
\item $(i<k)\vee(i=k \wedge j<l)$
\end{enumerate}
satisfying the inequality
\[
(2i+2k+j+l-12)-1/2(2j+2l+i+k-12)<0
\]
 
All possible solutions $\{i,j,k,l\}$ (in order) are:

\begin{eqnarray*}
\{0, 0, 0, 1\}, \{0, 0, 0, 2\}, \{0, 0, 0, 3\}, \{0, 0, 0, 4\}, \{0, 0, 0, 5\}, \{0, 0, 0, 6\}, \{0, 0, 1, 0\},\\
\{0, 0, 1, 1\}, \{0, 0, 1, 2\}, \{0, 0, 1, 3\}, \{0, 0, 1, 4\}, \{0, 0, 1, 5\}, \{0, 0, 2, 0\}, \{0, 0, 2, 1\}, \\
\{0, 0, 2, 2\}, \{0, 0, 2, 3\}, \{0, 0, 2, 4\}, \{0, 0, 3, 0\}, \{0, 0, 3, 1\}, \{0, 0, 3, 2\}, \{0, 0, 3, 3\}, \\
\{0, 1, 0, 2\}, \{0, 1, 0, 3\}, \{0, 1, 0, 4\}, \{0, 1, 0, 5\}, \{0, 1, 0, 6\}, \{0, 1, 1, 0\}, \{0, 1, 1, 1\}, \\
\{0, 1, 1, 2\}, \{0, 1, 1, 3\}, \{0, 1, 1, 4\}, \{0, 1, 1, 5\}, \{0, 1, 2, 0\}, \{0, 1, 2, 1\}, \{0, 1, 2, 2\}, \\
\{0, 1, 2, 3\}, \{0, 1, 2, 4\}, \{0, 1, 3, 0\}, \{0, 1, 3, 1\}, \{0, 1, 3, 2\}, \{0, 1, 3, 3\}, \{0, 2, 0, 3\}, \\
\{0, 2, 0, 4\}, \{0, 2, 0, 5\}, \{0, 2, 0, 6\}, \{0, 2, 1, 0\}, \{0, 2, 1, 1\}, \{0, 2, 1, 2\}, \{0, 2, 1, 3\}, \\
\{0, 2, 1, 4\}, \{0, 2, 1, 5\}, \{0, 2, 2, 0\}, \{0, 2, 2, 1\}, \{0, 2, 2, 2\}, \{0, 2, 2, 3\}, \{0, 2, 2, 4\}, \\
\{0, 2, 3, 0\}, \{0, 2, 3, 1\}, \{0, 2, 3, 2\}, \{0, 2, 3, 3\}, \{0, 3, 0, 4\}, \{0, 3, 0, 5\}, \{0, 3, 0, 6\}, \\
\{0, 3, 1, 0\}, \{0, 3, 1, 1\}, \{0, 3, 1, 2\}, \{0, 3, 1, 3\}, \{0, 3, 1, 4\}, \{0, 3, 1, 5\}, \{0, 3, 2, 0\}, \\
\{0, 3, 2, 1\}, \{0, 3, 2, 2\}, \{0, 3, 2, 3\}, \{0, 3, 2, 4\}, \{0, 3, 3, 0\}, \{0, 3, 3, 1\}, \{0, 3, 3, 2\}, \\
\{0, 3, 3, 3\}, \{0, 4, 0, 5\}, \{0, 4, 0, 6\}, \{0, 4, 1, 0\}, \{0, 4, 1, 1\}, \{0, 4, 1, 2\}, \{0, 4, 1, 3\}, \\
\{0, 4, 1, 4\}, \{0, 4, 1, 5\}, \{0, 4, 2, 0\}, \{0, 4, 2, 1\}, \{0, 4, 2, 2\}, \{0, 4, 2, 3\}, \{0, 4, 2, 4\}, \\
\{0, 4, 3, 0\}, \{0, 4, 3, 1\}, \{0, 4, 3, 2\}, \{0, 4, 3, 3\}, \{0, 5, 0, 6\}, \{0, 5, 1, 0\}, \{0, 5, 1, 1\}, \\
\{0, 5, 1, 2\}, \{0, 5, 1, 3\}, \{0, 5, 1, 4\}, \{0, 5, 1, 5\}, \{0, 5, 2, 0\}, \{0, 5, 2, 1\}, \{0, 5, 2, 2\}, \\
\{0, 5, 2, 3\}, \{0, 5, 2, 4\}, \{0, 5, 3, 0\}, \{0, 5, 3, 1\}, \{0, 5, 3, 2\}, \{0, 5, 3, 3\}, \{0, 6, 1, 0\}, \\
\{0, 6, 1, 1\}, \{0, 6, 1, 2\}, \{0, 6, 1, 3\}, \{0, 6, 1, 4\}, \{0, 6, 1, 5\}, \{0, 6, 2, 0\}, \{0, 6, 2, 1\}, \\
\{0, 6, 2, 2\}, \{0, 6, 2, 3\}, \{0, 6, 2, 4\}, \{0, 6, 3, 0\}, \{0, 6, 3, 1\}, \{0, 6, 3, 2\}, \{0, 6, 3, 3\}, 
\end{eqnarray*}
\begin{eqnarray*}
\{1, 0, 1, 1\}, \{1, 0, 1, 2\}, \{1, 0, 1, 3\}, \{1, 0, 1, 4\}, \{1, 0, 1, 5\}, \{1, 0, 2, 0\}, \{1, 0, 2, 1\}, \\
\{1, 0, 2, 2\}, \{1, 0, 2, 3\}, \{1, 0, 2, 4\}, \{1, 1, 1, 2\}, \{1, 1, 1, 3\}, \{1, 1, 1, 4\}, \{1, 1, 1, 5\}, \\
\{1, 1, 2, 0\}, \{1, 1, 2, 1\}, \{1, 1, 2, 2\}, \{1, 1, 2, 3\}, \{1, 1, 2, 4\}, \{1, 2, 1, 3\}, \{1, 2, 1, 4\}, \\
\{1, 2, 1, 5\}, \{1, 2, 2, 0\}, \{1, 2, 2, 1\}, \{1, 2, 2, 2\}, \{1, 2, 2, 3\}, \{1, 2, 2, 4\}, \{1, 3, 1, 4\}, \\
\{1, 3, 1, 5\}, \{1, 3, 2, 0\}, \{1, 3, 2, 1\}, \{1, 3, 2, 2\}, \{1, 3, 2, 3\}, \{1, 3, 2, 4\}, \{1, 4, 1, 5\}, \\
\{1, 4, 2, 0\}, \{1, 4, 2, 1\}, \{1, 4, 2, 2\}, \{1, 4, 2, 3\}, \{1, 4, 2, 4\}, \{1, 5, 2, 0\}, \{1, 5, 2, 1\}, \\
\{1, 5, 2, 2\}, \{1, 5, 2, 3\}, \{1, 5, 2, 4\}\,\,\,
\end{eqnarray*}

The question then is how to determine which coefficients in the defining polynomials of the generators need to vanish.

Note that we have introduced an ordering on the Pl\"{u}cker coordinates coming from the restrictions on $i,j,k$ and $l$. So, the first step is to look at the equation $m_{ijkl}=0$ for the first term $\{i,j,k,l\}$ in the list above, namely we look at the equation $m_{0001}=0$. It follows that either
\begin{enumerate}[(1)]
\item $f_{00}=g_{00}=0$ or
\item $g_{00}\neq 0$ or
\item $f_{00}\neq 0$
\end{enumerate}

Moreover, if $(2)$ (or $(3)$ by symmetry) holds, then taking $f'=f-\frac{f_{00}}{g_{00}}g$ we can assume $f_{00}=0$ and we must have $f_{01}=0$. 

The next step then is, at each of the cases above, to look at the next vanishing condition $m_{0002}=0$ coming from the second term $\{i,j,k,l\}$ in the list. Again there are three possibilities: Either $f_{00}=g_{00}=0$ or $g_{02}\neq 0$ or $f_{02}\neq 0$. 

We proceed in this manner until there are no more equations $m_{ijkl}=0$ to solve.

In fact our list tells us that $m_{00kl}$ vanish for all (appropriate) $0\leq k \leq 3$ and $0 \leq l \leq 6$. Thus, our algorithm tells us that if we are in the situation of case $(2)$, then one of the generators belongs to $\rangle x^kj^lz^{6-k-l} \langle$ for all $kl$ such that $m_{00kl}=0$.  And, by symmetry, we reach the same conclusion if $(3)$ holds. A similar reasoning applies to the next set of vanishing conditions $m_{01kl}=0$ and so on.

It is important to note, however, that at each step, when solving the equations $m_{ijkl}=0$ we have to take into account whether there are  or there are not previous conditions on the coefficients $f_{ij},g_{ij},f_{kl}$ and $g_{kl}$.

Following the sketched algorithm we obtain the desired geometric description of the pencil $\mathcal{P}$.
\end{proof}

Note that the same algorithm outlined above in the proof of Theorem \ref{gennotstable} can be applied whenever $\mathcal{P}$ is unstable (resp. not stable) and satisfies one of the vanishing conditions in anyone of the cases in Theorem \ref{unstable2} (resp. \ref{notstable2}). However, the computations involved are very lengthy and the assistance of a computer is needed. And even the corresponding statements as in Theorem \ref{gennotstable} require several pages to be presented.

The complete geometric description of the stability conditions in terms of equations for the generators is available upon request. And instead of exhibiting these tiresome results we will present next (without proofs) only those results that are essential in the study of Halphen pencils of index two (Section \ref{halphen}) and we will mostly focus on \textbf{proper} pencils:

\begin{defi}
A pencil $\mathcal{P} \in \mathscr{P}_6$ is called proper if any two curves on it intersect properly, meaning its base locus is zero dimensional, i.e. it consists of a finite number of points.
\end{defi}

\subsection{Equations associated do nonstability}

\begin{thm}
Let $\mathcal{P}\in \mathscr{P}_6$ be a proper pencil which contains a curve of the form $3L+C$, where $L$ is a line and $C$ is a cubic (possibly reducible). Then $\mathcal{P}$ is not stable if and only if there exist coordinates in $\mathbb{P}^2$ and generators $C_f$ and $C_g$ of $\mathcal{P}$ such that:

\begin{enumerate}[(a)]
\item $f\in \langle x^3z^3,x^3yz^2,x^3y^2z,x^3y^3,x^4z^2,x^4yz,x^4y^2,x^5z,x^5y,x^6 \rangle$ with $f_{30}\neq 0$ and $g$ satisfies

\vspace{0.5cm}
\begin{enumerate}[({a}1)]
\item  $g_{00}=\ldots=g_{05}=0,g_{10}=\ldots=g_{13}=0,g_{20}=g_{21}=0$ or
\item $g_{00}=\ldots=g_{04}=0,g_{10}=\ldots=g_{13}=0,g_{20}=g_{21}=g_{22}=g_{31}=g_{40}=0$ 
\end{enumerate}
\vspace{0.5cm}

\item $f\in \langle x^3yz^2,x^3y^2z,x^3y^3,x^4z^2,x^4yz,x^4y^2,x^5z,x^5y,x^6 \rangle$ with $f_{31}\neq 0$ and $g$ satisfies  

\vspace{0.5cm}
\begin{enumerate}[({b}1)]
\item $g_{00}=\ldots=g_{05}=0,g_{10}=g_{11}=g_{12}=0$ or
\item $g_{00}=\ldots=g_{04}=0,g_{10}=g_{11}=g_{12}=g_{20}=0$ or
\item  $g_{00}=\ldots=g_{03}=0,g_{10}=g_{11}=g_{12}=g_{20}=g_{21}=g_{30}=0$
\end{enumerate}
\vspace{0.5cm}

\item $f\in \langle x^3y^2z,x^3y^3,x^4z^2,x^4yz,x^4y^2,x^5z,x^5y,x^6 \rangle$ with $f_{32}\neq 0$ and
either

\vspace{0.5cm}
\begin{enumerate}[({c}1)]
\item $g$ satisfies $g_{00}=\ldots=g_{03}=0,g_{10}=g_{11}=0$ or
\item  $m_{ijkl}=0$ for $i,j,k,l$ (in order) in the list below
\[
\{0,3,4,0\}, \{1,2,4,0\}, \{2,1,4,0\}, \{3,0,4,0\}
\]
 and $g$ satisfies $g_{00}=g_{01}=g_{02}=g_{10}=g_{11}=g_{20}=0$ 
\end{enumerate}
\vspace{0.5cm}

\item $f\in \langle x^3y^3,x^4z^2,x^4yz,x^4y^2,x^5z,x^5y,x^6 \rangle$ with $f_{33}\neq 0$ and
either
\vspace{0.5cm}
\begin{enumerate}[({d}1)]
\item $g$ satisfies $g_{00}=\ldots=g_{03}=0,g_{10}=0$ or
\item $m_{ijkl}=0$ for $i,j,k,l$ (in order) in the list below
\[
\{0,3,4,0\}, \{1,1,4,0\}
\]
 and $g$ satisfies $g_{00}=g_{01}=g_{02}=g_{10}=0$ or
\item $m_{ijkl}=0$ for $i,j,k,l$ (in order) in the list below
\begin{eqnarray*}
\{0,2,4,0\},\{0,2,4,1\}, \{0,2,5,0\}, \{0,3,4,0\}, \{1,1,4,0\}, \{1,1,4,1\}, \{1,1,5,0\}, \\
\{2,0,4,0\}, \{2,1,4,0\}, \{3,0,4,0\}\,\,
\end{eqnarray*}
 and $g$ satisfies $g_{00}=g_{01}=g_{10}=0$
\end{enumerate}
\vspace{0.5cm}
\end{enumerate}
\label{nsreg}
\end{thm}

\begin{thm}
Let $\mathcal{P}\in \mathscr{P}_6$ be a proper pencil which contains a curve of the form $2L+Q$, where $L$ is a line and $Q$ is a quartic (possibly reducible). Then $\mathcal{P}$ is not stable if and only if there exist coordinates in $\mathbb{P}^2$ and generators $C_f$ and $C_g$ of $\mathcal{P}$ such that:

\begin{enumerate}[(a)]
\item $f\in \langle x^2z^4,x^2yz^3,x^2y^2z^2,x^2y^3z,x^2y^4, x^iy^jz^{6-i-j}  \rangle$, with $3\leq i \leq 6$, $0\leq j \leq 6, i+j \leq 6$ plus $f_{20}\neq 0$ and $g\in \langle y^6,xy^5,x^2y^4,x^3y^3,x^4y^2,x^5y,x^6 \rangle$ (in particular, $C_g$ is unstable)

\item $f\in \langle x^2yz^3,x^2y^2z^2,x^2y^3z,x^2y^4, x^iy^jz^{6-i-j}  \rangle$, with $3\leq i \leq 6,0\leq j \leq 6, i+j \leq 6$ plus $f_{21}\neq 0$ and $g$ satisfies 

\vspace{0.5cm}
\begin{enumerate}[({b}1)]
\item $g_{00}=\ldots=g_{05}=0,g_{10}=\ldots=g_{14}=0,g_{20}=g_{21}=g_{22}=g_{30}=g_{31}=0$ or
\item $g_{00}=\ldots=g_{04}=0,g_{10}=\ldots=g_{13}=0,g_{20}=g_{21}=g_{22}=g_{30}=g_{31}=g_{40}=0$ 
\end{enumerate}
\vspace{0.5cm}
in particular, $C_g$ is unstable.
\vspace{0.5cm}

\item $f\in \langle x^2y^2z^2,x^2y^3z,x^2y^4, x^iy^jz^{6-i-j}  \rangle$, with $3\leq i \leq 6,0\leq j \leq 6, i+j \leq 6$ plus $f_{22}\neq 0$ and either 

\vspace{0.5cm}
\begin{enumerate}[({c}1)]
\item $g$ satisfies $g_{00}=\ldots=g_{05}=0,g_{10}=\ldots=g_{13}=0,g_{20}=g_{21}=0$ or
\item $f_{30}=0$ and $g$ satisfies $g_{00}=\ldots=g_{04}=0,g_{10}=\ldots=g_{13}=0,g_{20}=g_{21}=g_{22}=g_{30}=0$ or
\item $m_{ijkl}=0$ for $i,j,k,l$ (in order) in the list below
\[
\{0,4,3,0\}, \{1,3,3,0\}, \{3,0,3,1\}, \{3,0,4,0\}
\]
and $g$ satisfies $g_{00}=\ldots=g_{03}=0,g_{10}=g_{11}=g_{12}=g_{20}=g_{21}=g_{22}=g_{30}=0$ 
\end{enumerate}
\vspace{0.5cm}
In particular, $(0:0:1)$ has multiplicity $\geq 3$ in  $C_g$.
\vspace{0.5cm}

\item $f\in \langle x^2y^3z,x^2y^4, x^iy^jz^{6-i-j}  \rangle$, with $3\leq i \leq 6,0\leq j \leq 6, i+j \leq 6$ plus $f_{23}\neq 0$ and either 

\vspace{0.5cm}
\begin{enumerate}[({d}1)]
\item $m_{ijkl}=0$ for $i,j,k,l$ (in order) in the list below
\[
\{0,5,3,0\}, \{1,3,3,0\}, \{2,1,3,0\}
\]
and $g$ satisfies $g_{00}=\ldots=g_{04}=0,g_{10}=g_{11}=g_{12}=g_{20}=0$ or 
\item $m_{ijkl}=0$ for $i,j,k,l$ (in order) in the list below
\[
\{0,4,3,0\}, \{0,4,3,1\}, \{0,5,3,0\}, \{1,3,3,0\}, \{2,1,3,0\}, \{2,1,3,1\}, \{2,2,3,0\}
\]
and $g$ satisfies $g_{00}=\ldots=g_{03}=0,g_{10}=g_{11}=g_{12}=g_{20}=0$ or 
\item $m_{ijkl}=0$ for $i,j,k,l$ (in order) in the list below
\begin{eqnarray*}
\{0,3,3,0\}, \{0,3,3,1\}, \{0,3,4,0\}, \{0,4,3,0\}, \{1,2,3,0\}, \{1,2,3,1\}, \{1,2,4,0\},\\
\{1,3,3,0\}, \{2,1,3,0\}, \{2,1,3,1\}, \{2,1,4,0\}, \{2,2,3,0\}, \{3,0,3,1\}, \{3,0,4,0\}\,\,
\end{eqnarray*}
and $g$ satisfies $g_{00}=g_{01}=g_{02}=g_{10}=g_{11}=g_{20}=0$ 
\end{enumerate}
\vspace{0.5cm}
In particular, $(0:0:1)$ has multiplicity $\geq 3$ in  $C_g$.
\vspace{0.5cm}

\item $f\in \langle x^2y^4, x^iy^jz^{6-i-j}  \rangle$, with $3\leq i \leq 6,0\leq j \leq 6, i+j \leq 6$ plus $f_{24}\neq 0$ and either 
\vspace{0.5cm}
\begin{enumerate}[({e}1)]
\item $m_{ijkl}=0$ for $i,j,k,l$ (in order) in the list below
\[
\{0,6,3,0\}, \{1,3,3,0\}, \{2,0,3,0\}
\]
 and $g$ satisfies $g_{00}=\ldots=g_{05}=0,g_{10}=g_{11}=g_{12}=0$ or 
\item $m_{ijkl}=0$ for $i,j,k,l$ (in order) in the list below
\begin{eqnarray*}
\{0,4,3,0\}, \{0,4,3,1\}, \{0,5,3,0\}, \{1,2,3,0\}, \{1,2,3,1\}, \{1,3,3,0\}, \{2,0,3,0\},\\
\{2,0,3,1\}, \{2,1,3,0\}\,\,
\end{eqnarray*}
 and $g$ satisfies $g_{00}=\ldots=g_{03}=0,g_{10}=g_{11}=0$ or 
\item $m_{ijkl}=0$ for $i,j,k,l$ (in order) in the list below
\begin{eqnarray*}
\{0,3,3,0\}, \{0,3,3,1\}, \{0,3,3,2\}, \{0,3,4,0\}, \{0,4,3,0\}, \{0,4,3,1\}, \{0,5,3,0\},\\
 \{1,2,3,0\}, \{1,2,3,1\}, \{1,2,4,0\}, \{1,3,3,0\}, \{2,0,3,0\}, \{2,0,3,1\}, \{2,0,3,2\}, \\
\{2,1,3,0\}, \{2,1,3,1\}, \{2,2,3,0\}\,\,
\end{eqnarray*}
 and $g$ satisfies $g_{00}=g_{01}=g_{02}=g_{10}=g_{11}=0$ or 
\item $m_{ijkl}=0$ for $i,j,k,l$ (in order) in the list below
\begin{eqnarray*}
\{0, 2, 3, 0\}, \{0, 2, 3, 1\}, \{0, 2, 3, 2\}, \{0, 2, 4, 0\}, \{0, 2, 4, 1\}, \{0, 2, 5, 0\}, \{0, 3, 3, 0\}, \\
\{0, 3, 3, 1\}, \{0, 3, 4, 0\}, \{0, 4, 3, 0\}, \{1, 1, 3, 0\}, \{1, 1, 3, 1\}, \{1, 1, 3, 2\}, \{1, 1, 4, 0\}, \\
\{1, 1, 4, 1\}, \{1, 1, 5, 0\}, \{1, 2, 3, 0\}, \{1, 2, 3, 1\}, \{1, 2, 4, 0\}, \{1, 3, 3, 0\}, \{2, 0, 3, 0\}, \\
\{2, 0, 3, 1\}, \{2, 0, 3, 2\}, \{2, 0, 4, 0\}, \{2, 0, 4, 1\}, \{2, 0, 5, 0\},  \{2, 1, 3, 0\}, \{2, 1, 3, 1\}, \\
\{2, 1, 4, 0\}, \{2, 2, 3, 0\}, \{3, 0, 3, 1\}, \{3, 0, 4, 0\}\,\,
\end{eqnarray*}
 and $g$ satisfies $g_{00}=g_{01}=g_{10}=0$ 
\end{enumerate}
\vspace{0.5cm}
\end{enumerate}
\label{nsreg1}
\end{thm}

\subsection{Equations associated to unstability}

\begin{thm}
A pencil $\mathcal{P}\in \mathscr{P}_6$ will satisfy the vanishing conditions in case $1$ of Theorem \ref{unstable2} if and only if we can find coordinates in $\mathbb{P}^2$ and generators $C_f$ and $C_g$ of $\mathcal{P}$ such that 
\begin{enumerate}[{Case} 1]
\item $f \in \langle x^5z,x^5y,x^6 \rangle$ and $g$ is arbitrary
\item $f \in \langle x^4y^2,x^5z,x^5y,x^6 \rangle$ and $g\in \rangle z^6,yz^5,y^2z^4\langle$ 
\item $f \in \langle x^4yz,x^4y^2,x^5z,x^5y,x^6 \rangle$ and $g\in \rangle z^6,yz^5,y^2z^4,y^3z^3\langle$ 
\item $f \in \langle x^4z^2,x^4yz,x^4y^2,x^5z,x^5y,x^6 \rangle$ and $g\in \rangle z^6,yz^5,y^2z^4,y^3z^3,y^4z^2\langle$ 
\item $f \in \langle x^3y^3,x^4z^2,x^4yz,x^4y^2,x^5z,x^5y,x^6 \rangle$ 
\subitem and $g\in \rangle z^6,yz^5,y^2z^4,y^3z^3,y^4z^2,y^5z\langle$ 
\item $f \in \langle x^3y^2z,x^3y^3,x^4z^2,x^4yz,x^4y^2,x^5z,x^5y,x^6 \rangle$ 
\subitem and $g\in \rangle z^6,yz^5,y^2z^4,y^3z^3,y^4z^2,y^5z,y^6,xz^5,xyz^4,xy^2z^3\langle$ 
\item $f \in \langle x^3yz^2,x^3y^2z,x^3y^3,x^4z^2,x^4yz,x^4y^2,x^5z,x^5y,x^6 \rangle$ 
\subitem and $g\in \rangle z^6,yz^5,y^2z^4,y^3z^3,y^4z^2,y^5z,y^6,xz^5,xyz^4,xy^2z^3,xy^3z^2\langle$ 
\item $f \in \langle x^3z^3,x^3yz^2,x^3y^2z,x^3y^3,x^4z^2,x^4yz,x^4y^2,x^5z,x^5y,x^6 \rangle$ 
\subitem and $g\in \rangle z^6,yz^5,y^2z^4,y^3z^3,y^4z^2,y^5z,y^6,xz^5,xyz^4,xy^2z^3,xy^3z^2,xy^4z,xy^5\langle$ 
\item $f \in \langle x^2y^4,x^3z^3,x^3yz^2,x^3y^2z,x^3y^3,x^4z^2,x^4yz,x^4y^2,x^5z,x^5y,x^6 \rangle$ 
\subitem and $g\in \rangle z^6,yz^5,y^2z^4,y^3z^3,y^4z^2,y^5z,y^6,xz^5,xyz^4,xy^2z^3,xy^3z^2,xy^4z,x^2z^4\langle$ 
\end{enumerate}
\label{equns}
\end{thm}

\begin{thm}
Let $\mathcal{P}\in \mathscr{P}_6$ be a proper pencil which contains a curve of the form $4L+Q$, where $L$ is a line and $Q$ is a conic (possibly reducible). Then $\mathcal{P}$ is unstable if and only if there exist coordinates in $\mathbb{P}^2$ and generators $C_f$ and $C_g$ of $\mathcal{P}$ such that:
\begin{enumerate}[(a)]
\item $f\in \langle x^4z^2,x^4yz,x^4y^2,x^5z,x^5y,x^6 \rangle$ plus $f_{40}\neq 0$ and either $g$ satisfies

\vspace{0.5cm}
\begin{enumerate}[({a}1)]
\item $g_{00}=\ldots=g_{04}=0$ or
\item $g_{00}=\ldots=g_{03}=0,g_{10}=g_{11}=g_{12}=g_{20}=0$ (in particular, $(0:0:1)$ has multiplicity $\geq 3$ in  $C_g$.).
\end{enumerate}
\vspace{0.5cm}

\item $f\in \langle x^4yz,x^4y^2,x^5z,x^5y,x^6 \rangle$ plus $f_{41}\neq 0$ and $g$ satisfies $g_{00}=\ldots=g_{03}=0$.

\item $f\in \langle x^4y^2,x^5z,x^5y,x^6 \rangle$ plus $f_{42}\neq 0$ and $g$ satisfies $g_{00}=g_{01}=g_{02}=0$.
\end{enumerate}
\label{unsreg}
\end{thm}

\begin{thm}
Let $\mathcal{P}\in \mathscr{P}_6$ be a proper pencil which contains a curve of the form $3L+C$, where $L$ is a line and $C$ is a cubic (possibly reducible). Then $\mathcal{P}$ is unstable if and only if there exist coordinates in $\mathbb{P}^2$ and generators $C_f$ and $C_g$ of $\mathcal{P}$ such that:

\begin{enumerate}[(a)]
\item $f\in \langle x^3z^3,x^3yz^2,x^3y^2z,x^3y^3,x^4z^2,x^4yz,x^4y^2,x^5z,x^5y,x^6 \rangle$ plus $f_{30}\neq 0$ and $g$ satisfies 
\[
g_{00}=\ldots=g_{05}=0,g_{10}=\ldots=g_{14}=0,g_{20}=g_{21}=g_{22}=0
\]
In particular, $(0:0:1)$ has multiplicity $\geq 3$ in  $C_g$.
\item $f\in \langle x^3yz^2,x^3y^2z,x^3y^3,x^4z^2,x^4yz,x^4y^2,x^5z,x^5y,x^6 \rangle$ plus $f_{31}\neq 0$ and $g$ satisfies 
\[
g_{00}=\ldots=g_{04}=0,g_{10}=\ldots=g_{13}=0,g_{20}=g_{21}=0
\]
In particular, $(0:0:1)$ has multiplicity $\geq 3$ in  $C_g$.
\item $f\in \langle x^3y^2z,x^3y^3,x^4z^2,x^4yz,x^4y^2,x^5z,x^5y,x^6 \rangle$ plus $f_{32}\neq 0$ and $g$ satisfies 
\[
g_{00}=\ldots=g_{04}=0,g_{10}=g_{11}=g_{12}=0
\]
\item $f\in \langle x^3y^2z,x^3y^3,x^4yz,x^4y^2,x^5z,x^5y,x^6 \rangle$ plus $f_{32}\neq 0$ and $g$ satisfies
\[
g_{00}=\ldots=g_{03}=0,g_{10}=g_{11}=g_{12}=g_{20}=0
\]
In particular, $(0:0:1)$ has multiplicity $\geq 3$ in $C_g$.
\item $f\in \langle x^3y^3,x^4z^2,x^4yz,x^4y^2,x^5z,x^5y,x^6 \rangle$ plus $f_{33}\neq 0$ and $g$ satisfies
\[
g_{00}=\ldots=g_{04}=0,g_{10}=g_{11}=0
\]
\item $f\in \langle x^3y^3,x^4yz,x^4y^2,x^5z,x^5y,x^6 \rangle$ plus $f_{33}\neq 0$ and $g$ satisfies
\[
g_{00}=g_{01}=g_{02}=g_{10}=g_{11}=0
\]
\end{enumerate}
\label{unsreg1}
\end{thm}

\begin{thm}
Let $\mathcal{P}\in \mathscr{P}_6$ be a proper pencil which contains a curve of the form $2L+Q$, where $L$ is a line and $Q$ is a quartic (possibly reducible). If $\mathcal{P}$ is unstable then there exist coordinates in $\mathbb{P}^2$ and generators $C_f$ and $C_g$ of $\mathcal{P}$ such that:
\[
f\in \langle x^2z^4,x^2yz^3,x^2y^2z^2,x^2y^3z,x^2y^4, x^iy^jz^{6-i-j}  \rangle
\]
with $3\leq i \leq 6,0\leq j \leq 6, i+j \leq 6$ plus $f_{2j}\neq 0$ for some $j=0,\ldots,4$ and $(0:0:1)$ has multiplicity $\geq 3$ in $C_g$.
\label{unsreg2}
\end{thm}

\section{Stability of Halphen pencils of index two}
\label{halphen}

We now combine the results from Sections \ref{sc} and \ref{gd},  together with the results we have previously obtained in \cite{azconstr} and \cite{azstabd}, in order to completely characterize the stability of Halphen pencils of index two. Similar to the work of Miranda we provide a geometric description in terms of the types of singular fibers of the corresponding rational elliptic surface. 

The relevant definitions are as follows:

\begin{defi}
A \textbf{rational elliptic surface}  consists of a smooth and projective rational surface $Y$ together with a fibration  (a surjective proper flat morphism) $f:Y\to \mathbb{P}^1$ such that the generic fiber is a smooth genus one curve and there are no $(-1)-$curves in any fiber.  
\end{defi}

\begin{defi}
Given $f:Y\to \mathbb{P}^1$ as above, we define the \textbf{index} of the fibration as the positive generator of the ideal $\{D\cdot Y_{\eta} \,\,; \,\, D\in \text{Pic}(Y)\} \trianglelefteq \mathbb{Z}$, where $Y_{\eta}$ is a generic fiber. 
\end{defi}

Any rational elliptic surface has finitely many singular fibers, and these have been classified by Kodaira and N\'{e}ron \cites{kodaira1, kodaira2, neron}. Table \ref{fibers}  below gives the full classification. Over a field of characteristic zero, any multiple fiber is of type $I_n$ for some $n\geq 0$ \cite{dc}*{Proposition 5.1.8} .

{\renewcommand{\arraystretch}{1.5}

\begin{table}[H]
\centering
\begin{tabular}{|c  | c | c |}
\hline 
\bf{Kodaira Type} & \bf{Number of Components} & \bf{Dual Graph}\\
\hline 
$I_0$ & 1 (smooth) & $\bullet$ \\
$I_1$ & 1 (with a node) & $\bullet$ \\
$I_n$ & $n\geq 2$ & $\tilde{A}_{n-1}$\\
$II$ & 1 (with a cusp) & $\bullet$ \\
$III$ & 2 & $\tilde{A}_1$ \\
$IV$ & 3 & $\tilde{A}_2$ \\
$I_n^*$ & $n+5$ & $\tilde{D}_{4+n}$\\
$IV^*$ & $7$ & $\tilde{E}_6$\\
$III^*$ & $8$ & $\tilde{E}_7$\\
$II^*$ & $9$ & $\tilde{E}_8$\\
\hline
\end{tabular}
\caption{Kodaira's Classification}
\label{fibers}
\end{table}
}

\begin{defi}
A \textbf{Halphen pencil of index m}  is a pencil of plane curves of degree $3m$  with nine (possibly infinitely near) base points of multiplicity $m$. 
\label{defHalphen} 
\end{defi}

The correspondence between Halphen pencils and rational elliptic surfaces is given by the following Proposition:

\begin{prop}[\cite{dc}*{Theorem 5.6.1}, \cite{fuji90}*{[Main Theorem 2.1]}]
Let $f:Y \to \mathbb{P}^1$ be a rational elliptic surface of index $m$ and let $F$ be a choice of  a fiber of $f$, then there exists a birational map $\pi: Y \to \mathbb{P}^2$ so that $f\circ \pi^{-1}$ is a Halphen pencil of index $m$ and, moreover, $B\doteq \pi(F)$  is a plane curve of degree $3m$.

Conversely, given a Halphen pencil of index $m$, taking the minimal resolution of its base points we obtain a rational elliptic surface of index $m$.
\label{Halphen}
\end{prop}

In particular, any Halphen pencil of index $m$ contains exactly one multiple cubic $mC$, which corresponds to the unique multiple fiber in the associated rational elliptic surface \cite{dc}*{Proposition 5.61,(iii)}. Thus any Halphen pencil $\mathcal{P}$ of index $m$ can be written in the following form $\lambda(B)+\mu(mC)=0$, where the curve $B$ corresponds to some (non-multiple) fiber of $Y$ that we denote by $F$.

With these notations in mind we now focus in the case $m=2$. In Section \ref{suf} we give some sufficient conditions for the (semi)stability of $\mathcal{P}$ and in Section \ref{nec} we describe some necessary conditions. A complete characterization is then obtained by further studying the stability of $\mathcal{P}$ when the fiber $F$ is of type $II^*,III^*$ or $IV^*$ (Sections \ref{seciistar}, \ref{seciiistar} and \ref{secivstar}).

\subsection{Sufficient conditions for (semi)stability}
\label{suf}

We first establish sufficient conditions for the (semi)stability of a Halphen pencil of index two. We will see that the log canonical thresholds (see e.g. \cite{singpairs}*{Section 8}) of the pairs $(\mathbb{P}^2,2C)$ and $(\mathbb{P}^2,B)$ plays a fundamental role.

\begin{thm}
If $\mathcal{P}$ is not stable, then $Y$  contains a non-reduced fiber.
\label{Hnotstable}
\end{thm}

\begin{proof}
Since $lct(\mathbb{P}^2,2C)=\frac{1}{2}$ \cite{azconstr}*{Proposition 4.9}, we conclude from \cite{azstabd}*{Theorem 1.1}, with $\alpha=\frac{1}{2}$,  that if the pencil $\mathcal{P}$ is not stable, then $\mathcal{P}$ contains a curve $B$ such that $lct(\mathbb{P}^2,B)\leq \frac{1}{2}$. By \cite{azconstr}*{Proposition 4.15} this implies the corresponding rational elliptic surface $Y \to \mathbb{P}^1$ contains a non-reduced fiber $F$.
\end{proof}

\begin{rmk}
A completely analogous argument in fact shows the statement of Theorem \ref{Hnotstable} is true for Halphen pencils of any index. 
\end{rmk}

\begin{thm}
If $\mathcal{P}$ is unstable, then $Y$  contains a fiber of type $II^*,III^*$ or $IV^*$.
\label{Hunstable}
\end{thm}

\begin{proof}
The proof is very similar to the proof of Theorem \ref{Hnotstable}. Since we know $lct(\mathbb{P}^2,2C)=\frac{1}{2}$ \cite{azconstr}*{Proposition 4.9}, we conclude from \cite{azstabd}*{Theorem 1.1}, by taking $\alpha=\frac{1}{2}$,  that if the pencil $\mathcal{P}$ is unstable, then $\mathcal{P}$ contains a curve $B$ such that $lct(\mathbb{P}^2,B)< \frac{1}{2}$. Thus, \cite{azconstr}*{Propositions 4.15 and 4.16} imply $Y$ contains a a fiber of type $II^*,III^*$ or $IV^*$.
\end{proof}

\subsection{Necessary conditions for (semi)stability}
\label{nec}

The next step is to obtain necessary conditions. Again, the log canonical thresholds of the pairs $(\mathbb{P}^2,2C)$ and $(\mathbb{P}^2,B)$ will be fundamental to our analysis. 

When $C$ is smooth and $B$ is semistable we prove:

\begin{prop}
If $C$ is smooth and all curves in $\mathcal{P}$ are stable except (possibly) for one curve that is semistable, then $\mathcal{P}$ is stable.
\end{prop}

\begin{proof}
It follows from \cite{azstabd}*{Theorem 1.5} and the fact that $2C$ is stable \cite{shah}.
\end{proof}

\begin{cor}
If $C$ is smooth, $F$ is of type $II^*,III^*$ or $IV^*$ and $B\doteq \pi(F)$ is semistable, then $\mathcal{P}$ is stable.
\label{badfiberstable}
\end{cor}

\begin{proof}
From the classification in \cite{perssonlist} we know that any other fiber of $Y$ is reduced. By \cite{azconstr}*{Propositions 4.14 and 4.15} we also know that all other curves in $\mathcal{P}$ are reduced and have log canonical threshold greater than $1/2$. As observed in  \cite{hacking} and \cite{kimlee}, this implies  all the curves in $\mathcal{P}$ are stable except for one curve that is semistable.
\end{proof}

\begin{cor}
If $C$ is smooth and $Y$ contains exactly one fiber $F$ of type $I_n^*, n\leq 4$, then $\mathcal{P}$ is stable.
\end{cor}

\begin{proof}
Again, from the classification in \cite{perssonlist} we know that any other fiber of $Y$ is reduced. Since the curve $B$ is such that $lct(\mathbb{P}^2,B)\geq 1/2$, hence it is semistable \cites{hacking, kimlee}, we can argue as in the proof of Corollary \ref{badfiberstable} to conclude all the curves in $\mathcal{P}$ are stable except (possibly) for one curve that is semistable.
\end{proof}

\begin{thm}
If $Y$ contains two fibers of type $I_0^*$, then $\mathcal{P}$ is strictly semistable if and only if there exists a one-parameter subgroup $\lambda$ (and coordinates in $\mathbb{P}^2$) such that the two curves corresponding to the fibers of type $I_0^*$  are both non-stable with respect to this $\lambda$.
\label{twoi0star}
\end{thm}

\begin{proof}
By \cite{azconstr}*{Proposition 4.16}, if $F$ is a fiber of type $I_0^*$, then the corresponding plane curve $B$ is such that $lct(\mathbb{P}^2,B)\geq \frac{1}{2}$, hence it is semistable \cites{hacking, kimlee}. The result then follows from \cite{azstabd}*{Theorem 1.6}. Note that from the topological Euler characteristic of $Y$ we know $C$ has to be smooth, hence stable \cite{shah}.
\end{proof}

And when $C$ is singular we prove:

\begin{thm}
If $C$ is singular and $Y$ contains exactly one fiber $F$ of type $I_n^*,n\leq 4$, then $\mathcal{P}$ is strictly semistable if and only if there exists a one-parameter subgroup $\lambda$ (and coordinates in $\mathbb{P}^2$) such that $2C$ and $B=\pi(F)$ are both non-stable with respect to this $\lambda$.
\end{thm}

\begin{proof}
Since both $2C$ and $B$ are semistable  and all other curves in $\mathcal{P}$ are stable, the result follows from \cite{azstabd}*{Theorem 1.6}.
\end{proof}

\begin{thm}
If $C$ is singular, $Y$ contains a fiber $F$ of type $II^*,III^*$ and $IV^*$ and the curve $B=\pi(F)$ is semistable, then $\mathcal{P}$ is strictly semistable if and only if there exists a one-parameter subgroup $\lambda$ (and coordinates in $\mathbb{P}^2$) such that $2C$ and $B$ are both non-stable with respect to this $\lambda$.
\label{2sssbad}
\end{thm}

\begin{proof}
Again,  the result follows from \cite{azstabd}*{Theorem 1.6} because both $2C$ and $B$ are semistable  and all other curves in $\mathcal{P}$ are stable.
\end{proof}

Finally, in order to complete our description, we need to study the stability of $\mathcal{P}$ when $F$ is a fiber of type $II^*,III^*$ or $IV^*$.

\subsection{The stability of $\mathcal{P}$ when $F$ is of type $II^*$}
\label{seciistar}

When $F$ of type $II^*$, then \cite{azconstr}*{Theorem 5.15} tells us $B$ can only be realized by one of the following plane curves:
\begin{enumerate}[(i)]
\item a triple conic 
\item a nodal cubic and an inflection line, with the line taken with multiplicity three 
\item two triples lines
\item a conic and a tangent line, with the line taken with multiplicity four 
\item a line with multiplicity five and another line 
\end{enumerate}

If $B$ is a triple conic, then $B$ is strictly semistable \cite{shah}. In this case, if $C$ is smooth, then $\mathcal{P}$ is stable (Corollary \ref{badfiberstable}) and if $C$ is singular, then $\mathcal{P}$ is strictly semistable if and only if there exists a one-parameter subgroup $\lambda$ (and coordinates in $\mathbb{P}^2$) such that $2C$ and $B$ are both non-stable with respect to this $\lambda$ (Theorem \ref{2sssbad}).

When $B$ is one of the curves in $(ii),(iii),(iv)$ or $(v)$ then we can use the explicit constructions obtained in \cite{azconstr} to conclude $\mathcal{P}$ is unstable.

\begin{prop}
If $Y$ contains a fiber of type $II^*$ and $\mathcal{P}$ contains a curve consisting of a triple line and a nodal cubic, then $\mathcal{P}$ is unstable.
\label{triple+cubicuns}
\end{prop}

\begin{proof}
Let $\mathcal{P}$ and $Y$ be as above. One can show that the line is an inflection line of both the nodal cubic and $C$, which is smooth \cite{azconstr}*{Example 7.57}.

In particular, we can find coordinates in $\mathbb{P}^2$ so that the curve $B$ has equation $x^3(xz^2-y^2(y+x))=0$ and $C$ is given by $x^2y+xz^2-y^3-xy^2=0$. Then the Pl\"{u}cker coordinates of $\mathcal{P}$ with respect to these coordinates satisfy the conditions in Case (1) of Theorem \ref{unstable2} and we conclude $\mathcal{P}$ is unstable. Alternatively, we can easily check the equations for $B$ and $2C$ belong to Case 4 of Theorem \ref{equns}. 

\end{proof}

\begin{prop}
If $Y$ contains a fiber of type $II^*$ and $\mathcal{P}$ contains a curve consisting of two triple lines, then $\mathcal{P}$ is unstable.
\label{2triplelinesuns}
\end{prop}

\begin{proof}
Let $\mathcal{P}$ and $Y$ be as above. One can show that one of the lines is an inflection line of $C$ and the other line must be tangent to the cubic with multiplicity two \cite{azconstr}*{Example 7.56}.

In particular, we can find coordinates in $\mathbb{P}^2$ so that $B$ is given by $x^3y^3=0$ and $C$ is given by $z^2x-y(y-x)(y-\alpha\cdot x)=0$, where $\alpha \in \mathbb{C}\backslash \{0,1\}$. Then the Pl\"{u}cker coordinates of $\mathcal{P}$ with respect to these coordinates satisfy the conditions in Case (1) of Theorem \ref{unstable2} and we conclude $\mathcal{P}$ is unstable. Alternatively, we can easily check the equations for $B$ and $2C$ belong to Case 5 of Theorem \ref{equns}.
\end{proof}

\begin{prop}
If $Y$ contains a fiber of type $II^*$ and $\mathcal{P}$ contains a curve consisting of a conic and a tangent line, with the line taken with multiplicity four, then $\mathcal{P}$ is unstable.
\label{line4uns}
\end{prop}

\begin{proof}
Let $\mathcal{P}$ and $Y$ be as above. One can show that $C$ must be tangent to the conic (resp. the line) at the point $Q\cap L$ with multiplicity six (resp. two) as in \cite{azconstr}*{Example 7.58}.

In particular, we can find coordinates in $\mathbb{P}^2$ so that $B$ is given by the zeros of $x^4(y^2+xz)$ and $C$ is given by $f=\sum f_{ij}x^iy^jz^{6-i-j}=0$, with $f_{00}=f_{01}=f_{02}=0$. Thus, the Pl\"{u}cker coordinates of $\mathcal{P}$ with respect to these coordinates satisfy the conditions in Case (1) of Theorem \ref{unstable2} and we conclude $\mathcal{P}$ is unstable. Alternatively, we can easily check the equations for $B$ and $2C$ belong to Case 2 of Theorem \ref{equns}. 
\end{proof}

\begin{prop}
If $Y$ contains a fiber of type $II^*$ and $\mathcal{P}$ contains a curve consisting of a line with multiplicity five and another line, then $\mathcal{P}$ is unstable.
\label{line5uns}
\end{prop}

\begin{proof}
Let $B\in\mathcal{P}$ be the curve consisting of a line with multiplicity five and another line. We can choose coordinates so that $B$ is the curve $x^5(x-z)=0$ and $C$ is the cubic $y^2z=x(x-z)(x-\alpha \cdot z)$ for some $\alpha \in \mathbb{C}\backslash \{0,1\}$ \cite{azconstr}*{Example 7.59}. Then the Pl\"{u}cker coordinates of $\mathcal{P}$ satisfy the vanishing conditions of Case (1) in Theorem \ref{unstable2}. Or, yet, we can easily check the equations for $B$ and $2C$ belong to Case 1 of Theorem \ref{equns}. 
\end{proof}

Combining Propositions \ref{triple+cubicuns} through \ref{line5uns} and \cite{azconstr}*{Theorem 5.15} we obtain the following characterization when $F$ is of type $II^*$:

\begin{thm}
If $Y$ contains a fiber $F$ of type $II^*$ and $B\doteq \pi(F)$ is not a triple conic, then $\mathcal{P}$ is unstable.
\end{thm}

\subsection{The stability of $\mathcal{P}$ when $F$ is of type $III^*$}
\label{seciiistar}

We now consider the case when $F$ is of type $III^*$.

From \cite{azconstr}*{Theorem 5.16} the curve $B$ can only be realized by one of the following plane curves:
\begin{enumerate}[(i)]
\item a double line, a cubic and another line 
\item a double conic and another conic (semistable)
\item a triple conic (semistable) 
\item two triple lines 
\item a triple line, a double line and another line 
\item a triple line, a conic and a line 
\item a triple line and a cubic 
\item a conic and a line, with the line taken with multiplicity four 
\item a line with multiplicity four and two other lines  
\end{enumerate}

If $B$ is semistable there are two possibilities: either $C$ is smooth, in which case $\mathcal{P}$ is stable (Corollary \ref{badfiberstable}); or $C$ is singular and then $\mathcal{P}$ is strictly semistable if and only if there exists a one-parameter subgroup $\lambda$ (and coordinates in $\mathbb{P}^2$) such that $2C$ and $B$ are both non-stable with respect to this $\lambda$ (Theorem \ref{2sssbad}).

When $B$ is unstable we can use the explicit constructions obtained in \cite{azconstr} to conclude $\mathcal{P}$ is strictly semistable.

\begin{prop}
If $Y$ contains a fiber $F$ of type $III^*$ and $B\doteq \pi(F)$ consists of a triple line, a double line and another line in general position, then $\mathcal{P}$ is not stable.
\label{iiistarns}
\end{prop}

\begin{proof}
Let $\mathcal{P}$ and $Y$ be as above. One can find coordinates in $\mathbb{P}^2$ as in \cite{azconstr}*{Example 7.49} so that the Pl\"{u}cker coordinates of $\mathcal{P}$ with respect to these coordinates satisfy the conditions in Case (3) of Theorem \ref{notstable2} and we conclude $\mathcal{P}$ is not stable. Alternatively, we can also apply Theorem \ref{nsreg}.
\end{proof}

\begin{lema}
If $\mathcal{P}$ contains a curve $B$ and a base point $P$ such that $\text{mult}_P(B)=6$, then $\mathcal{P}$ is not stable.
\label{m6}
\end{lema}

\begin{proof}
Since  $\text{mult}_P(2C)\geq 2$, the result follows from \cite{azstabd}*{Theorem 1.3}.
\end{proof}

\begin{prop}
If $Y$ contains a fiber $F$ of type $III^*$ and $B\doteq \pi(F)$ consists of a triple line, a double line and another line concurrent at a base point, then $\mathcal{P}$ is not stable.
\end{prop}

\begin{proof}
Let $\mathcal{P},Y$ and $B$ be as above. Then $\mathcal{P}$ contains a base point $P$ (the point where the 3 lines meet) such that $\text{mult}_P(B)=6$, and the result follows from Lemma \ref{m6}. 
\end{proof}

\begin{prop}
If $Y$ contains a fiber $F$ of type $III^*$ and $B\doteq \pi(F)$ consists of a double line, a nodal cubic and another line, then $\mathcal{P}$ is not stable.
\end{prop}

\begin{proof}
Let $\mathcal{P}$ and $Y$ be as above. One can find coordinates in $\mathbb{P}^2$ \cite{azconstr}*{Example 7.45} so that the Pl\"{u}cker coordinates of $\mathcal{P}$ with respect to these coordinates satisfy the conditions in Case (3) of Theorem \ref{notstable2} and we conclude $\mathcal{P}$ is not stable.
\end{proof}

\begin{prop}
If $Y$ contains a fiber $F$ of type $III^*$ and $B\doteq \pi(F)$ contains a line with multiplicity four, then $\mathcal{P}$ is not stable.
\end{prop}

\begin{proof}
If $B$ contains a line with multiplicity four, then we can find coordinates in $\mathbb{P}^2$ and generators of $\mathcal{P}$ which are given by equations as in Case $1$ of Theorem \ref{gennotstable}.
\end{proof}

\begin{prop}
If $Y$ contains a fiber $F$ of type $III^*$ and $B\doteq \pi(F)$ consists of a triple line and a nodal cubic, then $\mathcal{P}$ is not stable.
\end{prop}

\begin{proof}
We can find coordinates in $\mathbb{P}^2$ \cite{azconstr}*{Example 7.52} so that the Pl\"{u}cker coordinates of $\mathcal{P}$ with respect to these coordinates satisfy the conditions in Case (4) of Theorem \ref{notstable2} and we conclude $\mathcal{P}$ is not stable. Alternatively, we can also apply Theorem \ref{nsreg}.
\end{proof}

\begin{prop}
If $Y$ contains a fiber $F$ of type $III^*$ and $B\doteq \pi(F)$ consists of a triple line, a conic and another line,
then $\mathcal{P}$ is not stable.
\end{prop}

\begin{proof}
Let $\mathcal{P}$ and $Y$ be as above. We can find coordinates in $\mathbb{P}^2$ as in \cite{azconstr}*{Example 7.51} so that the Pl\"{u}cker coordinates of $\mathcal{P}$ with respect to these coordinates satisfy the conditions in Case (3) of Theorem \ref{notstable2} and we conclude $\mathcal{P}$ is not stable. Alternatively, we can also apply Theorem \ref{nsreg}.
\end{proof}

\begin{prop}
If $Y$ contains a fiber $F$ of type $III^*$ and $B\doteq \pi(F)$ consists of two triple lines,
then $\mathcal{P}$ is not stable.
\label{iiistarns1}
\end{prop}

\begin{proof}
It follows from Lemma \ref{m6}.
\end{proof}

Combining Propositions \ref{iiistarns} through \ref{iiistarns1} and \cite{azconstr}*{Theorem 5.16} we obtain:

\begin{thm}
If $Y$ contains a fiber $F$ of type $III^*$ and $B\doteq \pi(F)$ is unstable, then $\mathcal{P}$ is not stable.
\end{thm}

\begin{rmk}
Note that when $F$ is of type $III^*$ and $B\doteq \pi(F)$ is semistable we can refer to Corollary \ref{badfiberstable} and Theorem \ref{2sssbad}.
\end{rmk}

So the remaining question is: Can $\mathcal{P}$ be unstable? We will show the answer to this questions is no.

\begin{lema}
Let $\mathcal{P}$ be a Halphen pencil of index two containing a curve $B$ such that $B=4L+Q$, where $L$ is a line and $Q$ is a conic (possibly reducible). Letting $2C$ denote the unique multiple cubic in $\mathcal{P}$ we have that if $\mathcal{P}$ is unstable, then either
\begin{enumerate}[(i)]
\item $L$ is an inflection line of $C$ or 
\item $L$ is tangent to $C$ at a point where $L$ and $Q$ also intersect 
\end{enumerate}
\label{uns4line}
\end{lema}

\begin{proof}
It follows from Theorem \ref{unsreg}.
\end{proof}

\begin{prop}
If $Y$ contains a fiber $F$ of type $III^*$ and $B\doteq \pi(F)$ contains a line with multiplicity four,
then $\mathcal{P}$ is semistable.
\label{iiistar4line}
\end{prop}

\begin{proof}
If $\mathcal{P}$ were unstable, then $\mathcal{P}$ (and $B$) would be as in (i) or (ii) in Lemma \ref{uns4line}. In \cite{azconstr}*{Section 6} we show that this is not case for a fiber of type $III^*$.
\end{proof}

\begin{lema}
Let $\mathcal{P}$ be a Halphen pencil of index two containing a curve $B$ such that $B=3L+C'$, where $L$ is a line and $C'$ is a cubic (possibly reducible). Letting $2C$ denote the unique multiple cubic in $\mathcal{P}$ we have that if $\mathcal{P}$ is unstable, then either
\begin{enumerate}
\item $L$ is an inflection line of $C$  at a point where the intersection multiplicity of $L$ and $C'$ is $\geq 2$.
\item $L$ is tangent to $C$ at a point where the intersection multiplicity of $L$ and $C'$ is $3$.
\end{enumerate}
\label{uns3line}
\end{lema}

\begin{proof}
It follows from Theorem \ref{unsreg1}.
\end{proof}

\begin{prop}
If $Y$ contains a fiber $F$ of type $III^*$ and $B\doteq \pi(F)$ contains a triple line,
then $\mathcal{P}$ is semistable.
\label{iiistar3line}
\end{prop}

\begin{proof}
If $\mathcal{P}$ were unstable, then $\mathcal{P}$ (and $B$) would be as in (i) or (ii) in Lemma \ref{uns3line}. In \cite{azconstr}*{Section 6} we show that this is not case for a fiber of type $III^*$.
\end{proof}

\begin{prop}
If $Y$ contains a fiber $F$ of type $III^*$ and $B\doteq \pi(F)$ consists of a double line, a cubic and another line,
then $\mathcal{P}$ is semistable.
\label{iiistar2line}
\end{prop}

\begin{proof}
It follows from Theorem \ref{unsreg2}.
\end{proof}

\subsection{The stability of $\mathcal{P}$ when $F$ is of type $IV^*$}
\label{secivstar}

Finally, we describe the stability of $\mathcal{P}$ when $F$ is of type $IV^*$. We will show that either $\mathcal{P}$ is stable or $C$ is singular and $B$ is semistable, in which case we are in the situation of Theorem \ref{2sssbad}.

\begin{lema}
Let $\mathcal{P}$ be a Halphen pencil of index two containing a curve $B$ such that $B=3L+C'$, where $L$ is a line and $C'$ is a cubic (possibly reducible). Letting $2C$ denote the unique multiple cubic in $\mathcal{P}$ we have that if $\mathcal{P}$ is not stable, then either
\begin{enumerate}[(i)]
\item $L$ is an inflection line of $C$ or 
\item $L$ is tangent to $C$ at a point where $L$ and $C'$ also intersect or
\item there is a base point where $L$ and $C$ intersect and where the intersection multiplicity of $L$ and $C'$ is $3$ 
\end{enumerate}
\label{nstripleline}
\end{lema}

\begin{proof}
It follows from Theorem \ref{nsreg}.
\end{proof}

\begin{prop}
If $Y$ contains a fiber $F$ of type $IV^*$ and $B\doteq \pi(F)$ contains a triple line,
then $\mathcal{P}$ is stable.
\label{ivstartripleline}
\end{prop}

\begin{proof}
If $\mathcal{P}$ were not stable, then $\mathcal{P}$ (and $B$) would be as in (i),(ii) or (iii) in Lemma \ref{nstripleline}. In \cite{azconstr}*{Section 6} we show that this is not case for a fiber of type $IV^*$.
\end{proof}

\begin{lema}
Let $\mathcal{P}$ be a Halphen pencil of index two containing a curve $B$ such that $B=2L+Q$, where $L$ is a line and $Q$ is a quartic (possibly reducible). Letting $2C$ denote the unique multiple cubic in $\mathcal{P}$ we have that if $\mathcal{P}$ is not stable, then the intersection multiplicity of $L$ and $Q$ at some base point is $4$.
\label{nsdoubleline}
\end{lema}

\begin{proof}
It follows from Theorem \ref{nsreg1}.
\end{proof}

\begin{thm}
If $Y$ contains a fiber of type $IV^*$ and $\mathcal{P}$ is not stable, then $C$ is singular and $B$ is semistable.
\label{nsivstar}
\end{thm}

\begin{proof}
If $\mathcal{P}$ is not stable, then it follows from Corollary \ref{badfiberstable} that either $C$ is singular or $B$ is unstable. Now, if $B$ is unstable, then the results from \cite{azconstr}*{Sections 6 and 7} and \cite{shah}*{Section 2} together with Theorem \ref{ivstartripleline} tell us $B$ can only (possibly) be
\begin{enumerate}
\item a double line, a nodal cubic and another line (\cite{azconstr}*{Example 7.35}) or
\item two double lines and a conic (\cite{azconstr}*{Example 7.38})
\end{enumerate}
but in any case $\mathcal{P}$ is stable by Lemma \ref{nsdoubleline}.
\end{proof}

\section*{Acknowledgments}
I thank my advisor, Antonella Grassi, for the many helpful discussions and the numerous suggestions on earlier versions of this paper. I also would like to thank Eduardo Esteves for pointing out the reference on stability of foliations. This work is part of my PhD thesis and it was partially supported by a Dissertation Completion Fellowship at the University of Pennsylvania.

\bibliography{references}

\end{document}